\documentclass[reqno]{amsart}

\usepackage{upref,amsxtra,amscd,amsopn,amsbsy,amstext,amssymb,amsmath,amsthm}
\usepackage{bm}
\usepackage{amsfonts}
\usepackage{ulem}
\usepackage{mathrsfs}
\usepackage{braket}
\usepackage{color}
\usepackage[pdftex]{graphicx}

\newcommand{\Add}[1]{#1} 
\newcommand{\add}[1]{#1} 
\newcommand{\Erase}[1]{\if0{#1}\fi}

\makeatletter
\@addtoreset{equation}{section}

\makeatother

\newtheorem{theorem}{Theorem}[section]
\newtheorem{proposition}[theorem]{Proposition}

\newtheorem{lemma}[theorem]{Lemma}

\theoremstyle{definition}
\newtheorem{definition}[theorem]{Definition}
\newtheorem{remark}[theorem]{Remark}

\newcommand{\lm}{\lambda}
\newcommand{\R}{\mathbb{R}}
\newcommand{\Z}{\mathbb{Z}}

\newcommand{\N}{\mathbb{N}}
\newcommand{\vk}{\kappa}
\newcommand{\vp}{\varphi}

\newcommand{\va}{\alpha}
\newcommand{\vb}{\beta}
\newcommand{\vc}{\gamma}

\newcommand{\ve}{\varepsilon}

\newcommand{\vt}{\tau}

\newcommand{\pd}{\partial}

\newcommand{\W}{\mathcal{W}}
\newcommand{\LL}{\mathcal{L}}

\newcommand{\sn}{\mathrm{sn}}
\newcommand{\cn}{\mathrm{cn}}
\newcommand{\dn}{\mathrm{dn}}
\newcommand{\E}{\mathrm{E}}
\newcommand{\K}{\mathrm{K}}
\newcommand{\pinA}{\mathcal{A}_{l,L}}

\allowdisplaybreaks

\begin{document}

\title[Critical points of pinned elasticae]{
The critical points of the elastic energy among curves pinned at endpoints
}
\author[K.~Yoshizawa]{Kensuke Yoshizawa}
\email{kensuke.yoshizawa.s5@dc.tohoku.ac.jp}
\keywords{
Euler's elastica; boundary value problem; shooting method.
}
\subjclass[2010]{
49K15, 53A04, 34A05}

\date{\today}

\maketitle

\vskip-5mm
{\small
\hskip5mm\noindent
        Mathematical Institute,
        Tohoku University,  
        Aoba, Sendai  980-8578, Japan 
}

\begin{abstract}
In this paper we find curves minimizing the elastic  energy among curves whose length is fixed and whose ends are pinned.
Applying the shooting method, 
\Add{we can identify all critical points explicitly and determine which curve is the global minimizer.}
As a result we show that the critical points consist of wavelike elasticae and the
minimizers do not have 
\Add{any loops or interior inflection points.}
\end{abstract}

\section{Introduction} \label{section:1}
Let $\vc$ be a smooth planar curve. 
Then \textit{elastic  energy} for $\vc$ is given by
\[
\W(\vc):=\int_{\vc}\vk^2\,ds, 
\]
where $s$ denotes the arc length parameter and $\vk$ denotes the curvature.
The minimization problem for $\W$ is called \textit{Euler's elastica problem} and has been studied due to not only mathematical interest but also the importance of applications.
See e.g., \cite{Antman, lev, Love, Tru} for more details of the history and see e.g., \cite{CMM, CKS, DFLM, Mum} for \Erase{the}applications. 

For given constants $0<l<L$, we define 
\begin{align*}
\pinA := \Set{ \vc \in C^{\infty} \big([0,1];\R^2 \big) | 
\begin{array}{l}
\vc(0)=(0,0),  \ 
\vc(1)= (l,0), \\
\LL(\vc)=L, \ \ \ \ \ 
\vc'(t)\ne0 
\end{array}
},
\end{align*}
where $\LL(\vc)$ denotes the length of $\vc$.
Hereafter, we use both the original parameter $t\in[0,1]$ and the arc length \add{parameter} $s\in[0,L]$.
For a curve $\vc$\Erase{$\in \pinA$}, we denote its arc length reparameterization by $\tilde{\vc}$.
In this paper we are interested in
the minimization problem for $\W$ among curves belonging to $\pinA$.
According to \cite{AGP20}, critical points of  $\W$ in $\pinA$ are called \textit{pinned elasticae}
\Add{and} a way to \Erase{deduce}\Add{approximate} pinned elasticae by a numerical procedure is \Add{demonstrated}.
However, as mentioned in \cite{AGP20},
\begin{itemize}
\item[(A)] \textit{whether one can construct all critical points explicitly or not}
\end{itemize}
is \Add{an open problem.}
Although Linn\'{e}r \cite{Lin98} obtained some results \Erase{of}\Add{toward} (A), one needs to solve a complicated system so as to obtain curves explicitly.
Furthermore, to the best of the author's knowledge, the following problems are also open:
\begin{itemize}
\item[(B)] \Add{\textit{Are global minimizers of $\W$ in $\pinA$ unique?}}
\item[(C)] \textit{Which of the critical points is the global minimizer?}
\end{itemize}
The aim of this paper is to give \Erase{affirmative}answers to problems (A), (B) and (C).

Let $\K(p)$ and $\E(p)$ be the complete elliptic integrals of the first and second kind, respectively ($p\in(0,1)$ is the modulus) and let $\cn$ be the Jacobi elliptic function 
(see Section \ref{subsect:2.2} for definitions).
Theorem~\ref{thm:1.1} is concerned with problem (A):
\begin{theorem} \label{thm:1.1}
The set of critical points of $\W$ in $\pinA$ is
\[
\Set{
\vc \in \pinA | \tilde{\vc} \text{ is } \hat{\vc}_n^{+} \text{ or } \hat{\vc}_n^{-} \text{ or } \check{\vc}_n^{+} \text{ or }\check{\vc}_n^{-}
\quad  \text{for some}\ \  n\in\N\cup\{0\}
}.
\]
Here
\begin{itemize}
\item[(i)] $\hat{\vc}_n^{\pm}(s)=(\hat{X}_n(s), \pm\hat{Y}_n(s))$ for $s\in[0,L]$,
\begin{align*}
\hat{X}_n(s)&= 
2 \hat{p}^2 \int_0^s \cn\bigg( \frac{2(n+1)\K(\hat{p})}{L} t +\K(\hat{p}), \hat{p} \bigg)^2\,dt + \big(1-2\hat{p}^2 \big) s,   \\
\hat{Y}_n(s)&= -\frac{ \hat{p} L}{(n+1)\K(\hat{p})}
 \cn\bigg( \frac{2(n+1)\K(\hat{p})}{L} s+\K(\hat{p}), \hat{p} \bigg),
\end{align*}
where $\hat{p}$ is uniquely determined by the solution of $2\E(p)/\K(p) -1  = {l}/{L}$.
\vspace{0.5\baselineskip}
\item[(ii)] $\check{\vc}_n^{\pm}(s)=(\check{X}_n(s), \pm\check{Y}_n(s))$ for $s\in[0,L]$,
\begin{align*}
\check{X}_n(s)&= 
-2 \check{p}^2 \int_0^s \cn\bigg( \frac{2(n+1)\K(\check{p})}{L} t -\K(\check{p}), \check{p} \bigg)^2\,dt - \big(1-2\check{p}^2 \big) s,   \\
\check{Y}_n(s)&= \frac{ \check{p} L}{(n+1)\K(\check{p})}
 \cn\Big( \frac{2(n+1)\K(\check{p})}{L} s-\K(\check{p}), \check{p} \Big),
\end{align*}
where $\check{p}$ is uniquely determined by the solution of $-2\E(p)/\K(p) +1  = {l}/{L}$.
\end{itemize}
\end{theorem}

We remark that $\hat{\vc}_n^{-}$ and $\check{\vc}_n^{-}$ are obtained by the reflection of $\hat{\vc}_n^{+}$ and $\check{\vc}_n^{+}$ across the $x$-axis, respectively. 
The number $n\in \N\cup\{0\}$ in Theorem~\ref{thm:1.1} \Erase{implies}\Add{is} the number of inflection points of  $\hat{\vc}_n^{\pm}$ or $\check{\vc}_n^{\pm}$ in $0<x<l$ (see Figure~\ref{fig:1.1}).
In particular, from the formulas in Theorem~\ref{thm:1.1}, we infer that the curves $\check{\vc}_n^{\pm}$ have $(n+1)$-loops.
Moreover, \Erase{we can obtain the relation between $l/L$ and whether $\hat{\vc}_n^{\pm}$ can be represented as a graph}
\Add{we find that whether or not $\hat{\vc}_n^{\pm}$ can be represented as a graph depends on the ratio $l/L$}
(see Section~\ref{section:3} for these characterizations).

\begin{figure}[http]
\centering
\includegraphics[width=9.5cm]{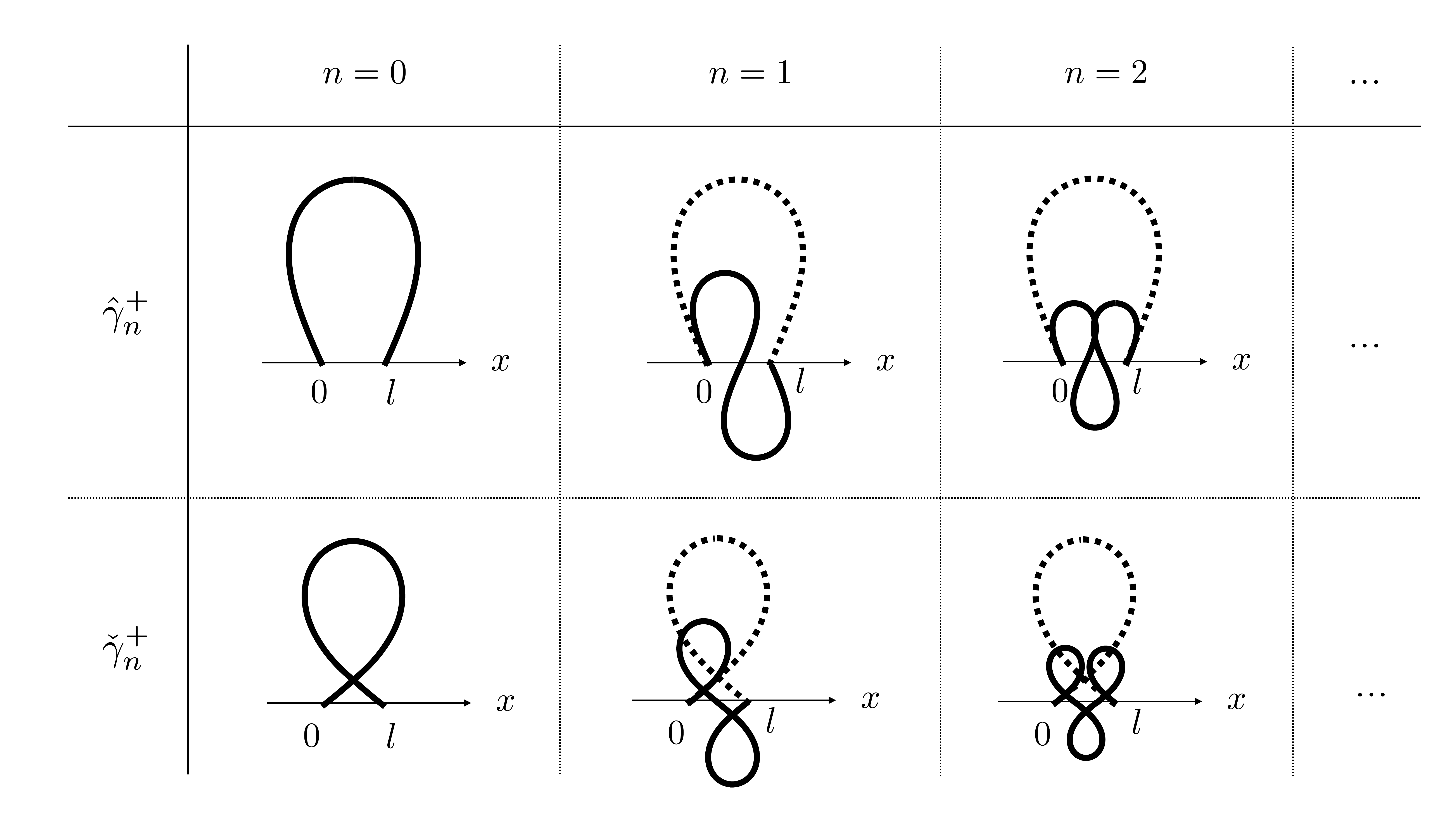} 
\caption{Critical points of $\W$ in $\pinA$ are given by Theorem~\ref{thm:1.1}.
According to \cite{Sin}, these curves are called wavelike elasticae.
}
\label{fig:1.1}
\end{figure}

Hence \Erase{thanks to}\Add{by} Theorem~\ref{thm:1.1}, not only the formulas but also the shapes of all critical points of $\W$ in $\pinA$ are deduced.
Furthermore, since we can compare the energy of each critical point (see Section~\ref{subsect:3.3}), we have:
\begin{theorem} \label{thm:1.2}
Let $0<l<L$ be arbitrary.
For the curves $\hat{\vc}_n^{\pm}$ and $\check{\vc}_n^{\pm}$ defined by Theorem~\ref{thm:1.1}, it holds that
\[
\W(\hat{\vc}_n^{\pm})= (n+1)^2\W(\hat{\vc}_0^{+}), \quad 
\W(\check{\vc}_n^{\pm})= (n+1)^2\W(\check{\vc}_0^{+}), \quad 
\W(\hat{\vc}_n^{+}) < \W(\check{\vc}_n^{+}),
\]
for $n\in\N\cup\{0\}$.
\end{theorem}

Moreover, we also obtain $|\W(\hat{\vc}_0^{+}) - \W(\check{\vc}_0^{+}) |\to0$ as $l\to0$.
Finally, problems (B) and (C) are solved as a corollary of Theorem~\ref{thm:1.2}:
\begin{theorem} \label{thm:1.3}
For each $0<l<L$, the global minimizer of $\W$ in $\pinA$ is $\hat{\vc}_0^{\pm}$, which implies
\begin{itemize}
\item[(i)] the uniqueness of global minimizers holds up to \Erase{the}reflection;
\item[(ii)] the global minimizer has no loop and \Erase{its inflection points exist only on endpoints.}\Add{the only inflection points are its endpoints.}
\end{itemize}
\end{theorem}


The minimization problem for $\W$ has been studied among various classes. 
Let $\vc$ be a critical point of $\W$ under the assumption that the length is prescribed.
Then, by the classical Lagrange multiplier method, the curvature $\vk$ of $\vc$ satisfies 
\begin{align} \label{EL}
 2\pd_{s}^2 \vk + \vk^3 -\lm\vk=0 
\end{align}
for some $\lm\in\R$. 
Euler \cite{Euler} derived the equation \eqref{EL} in 18th century and 
the curves whose curvature solves  \eqref{EL} are called \textit{Euler's elasticae}.
The shapes of Euler's elasticae are well known, \Erase{e.g., see}\Add{see e.g.,} \cite[Figure 11]{lev}.

The patterns of elasticae among closed curves are well understood.
It is shown in \cite{AKS13, LS_85, Sach12_1} that 
any critical point is an $n$-fold circle or an $n$-fold figure-of-eight, any local minimizer is an $n$-fold circle or the $1$-fold figure-of-eight, and a global minimizer is the $1$-fold circle. 
\Add{Moreover, critical points have been studied not only in $\R^2$ but also in higher-dimensional spaces, other manifolds and the hyperbolic space
(e.g. \cite{Koiso92, LS_JDG, LS_LMS, Sin}).}
In the case of open curves, however, there have been less results than the case of closed curves.
For example, clamped boundary value problems (the tangents at endpoints are prescribed) are considered:
Watanabe \cite{Watanabe2014} deduced the representation formulas \Erase{under a restriction of tangents at endpoints}\Add{for a special case}; 
Miura \cite{Miura20} revealed the shapes of minimizers in view of the phase transitions.
However, some problems such as \Erase{the uniqueness}\Add{uniqueness (B)} are still open.
One of difficulties for open curves is the treatment of boundary conditions.

Furthermore, another difficulty arises from the multiplier $\lm$.
In the case of closed curves, representation formulas have been obtained, taking into account the multiplier (e.g., \cite{MMY_09, MMY_13, Yana}).
In the case of open curves, the equation \eqref{EL} under Navier boundary conditions (the curvatures at endpoints are prescribed) has been solved in the case of $\lm=0$, see e.g. \cite{DG_07, DG_09, Man}.
To the best of the author's knowledge, however, \Erase{there is no result of representation formulas for open curves with handling the multiplier.}\Add{there are no results obtaining representation formulas for open curves with non-zero $\lm$, except for \cite{Watanabe2014}.}

In order to overcome these difficulties, we apply the shooting method to the problem.
The shooting method enables us to deal with the boundary conditions \Add{and} the multiplier simultaneously
(see Section~\ref{section:3}).
Consequently we obtain
the formula, \Erase{the characterization of curves (such as the shapes)}\Add{properties of curves (such as inflection points, the number of loops)}
and the uniqueness of minimizers, as we state in Theorems~\ref{thm:1.1} and \ref{thm:1.3}.

Another emphatic point which the formulas give is the relation between minimizers and \Erase{the number of inflection numbers (inflection means sign-changing points of the curvature).}\Add{the number of inflection points (points where the curvature changes sign).}
It is already known that the global minimizer has at most two inflection points by \Erase{different ways from this paper}\Add{different methods} (see \cite{Born, Sach08_1, Sach08_2, Sach14}).
Hence Theorem~\ref{thm:1.3} yields \Erase{another aspect to prove the relation}\Add{another point of view on the relation} between minimizers and inflection points.
The author expects that the method developed in this paper is applicable to the analysis of \eqref{EL} with Navier and clamped boundary conditions. 

\Add{
The elastic energy can be regarded as the one-dimensional version of the Willmore functional.
The Willmore functional is the integral of the squared mean curvature of surfaces, and its critical points are called Willmore surfaces. 
See e.g. \cite{Dall12, Nit, Scha10} for boundary value problems for Willmore surface and
see e.g. \cite{BDF10, DDG, DDW13, DFGS11, Eich17, Man18} for Willmore surfaces of revolution.
Recently, Mandel \cite{Man18} obtained explicit formulas for Willmore surfaces of revolution (with no constraints).
Our strategy for dealing with the Lagrange multiplier may be useful for extending the results of  \cite{Man18} to some constrained problems.}

This paper is organized as follows. 
In Sect.~\ref{section:2} we shall consider the Euler-Lagrange equation for $\W$ and collect the notations and facts related to the Jacobi Elliptic function.
In Sect.~\ref{section:3} we apply the shooting method to the Euler-Lagrange equation so that  Theorems~\ref{thm:1.1} and \ref{thm:1.2} are proved.



\section{Preliminaries} \label{section:2}

\subsection{Euler-Lagrange equation} \label{subsect:2.1}

The following equation \eqref{eq:EL} is famous for the Euler-Lagrange equation of the elastic  energy:

\begin{definition}
We call $\vc\in\pinA$ \textit{critical curve} if there exists a \add{constant} $\lambda\in\R$ such that the curvature $\vk$ of $\vc$ satisfies 
\begin{align}\label{eq:EL}
\pd_{s}^2\vk+\frac{1}{2}\vk^3-\frac{\lm}{2}\vk = 0 \quad \text{and} \quad \vk(0)=\vk(L)=0.
\end{align}
\end{definition}

\Add{We remark here that it is not restrictive to consider smooth curves only:
\begin{remark}\label{rem:2.1}
As the admissible set, 
\begin{align*}
H_{l,L} := \Set{ \vc \in H^2\big(0,1;\R^2 \big) | 
\begin{array}{l}
\vc(0)=(0,0),  \ 
\vc(1)= (0,l), \\
\LL(\vc)=L, \ \ \ \ \ 
\vc'(t)\ne0 
\end{array}
}.
\end{align*}
seems more suitable than $\pinA$.
Nevertheless we choose $\pinA$ as the admissible set since the regularity of $\vc\in H_{l,L}$ satisfying \eqref{eq:EL} can be always improved, up to a reparameterization (for details, see Appendix A).
\end{remark}
}

Hereafter we shall show that \eqref{eq:EL} holds if and only if $\vc$ is a critical point of $\W$ in $\pinA$ and shall deduce the related formulas.
Let $\vc$ be a curve belonging to $\pinA$ and $\tilde{\vc}(s)=(X(s), Y(s))$ be its arc length parameterization.
\add{For sufficiently small $\ve_0>0$, consider the variation $\{\vc_{\ve}\}_{|\ve|<\ve_0}$.
Set
}
\[ \vc_{\ve}(s) = (\phi(\ve, s), \psi(\ve,s))  \quad \text{for}\quad s\in [0,L], \]
\add{where $\phi$, $\psi \in C^{\infty}((-\ve_0, \ve_0) \times [0,L])$ and } 
\begin{itemize}
\item $\phi(0,s) =X(s)$, \  $\psi(0,s)=Y(s)$ \ for $s\in[0,L]$, 
\item $\phi(\ve,0)= \psi (\ve,0)= \psi(\ve,L)=0$, \  $\phi (\ve,L)=  l $.
\end{itemize}
We note here that $s$ does not generally give the arc length parameter for $\vc_{\ve}$.
Setting 
\[ \ell(\ve,s) = \sqrt{\pd_s \phi (\ve, s)^2 + \pd_s \psi (\ve, s)^2},
\]
we notice that 
\[ \vt(s) = \int_0^s \ell(\ve,\zeta) \,d\zeta \]
gives the arc length parameter of $\vc_{\ve}$.
We denote the unit tangent vector $\mathbf{t}$ and unit normal vector $\mathbf{n}$ of $\vc_{\ve}$  by 
\[ 
\add{\mathbf{t}(\ve,s)}=\ell(\ve,s)^{-1}\big(\pd_s \phi (\ve, s), \pd_s \psi (\ve, s)  \big), \quad 
\add{\mathbf{n}(\ve,s)}=   \left(
 \begin{array}{cc}
      0  & -1 \\
      1 & 0 
    \end{array} \right)
    \add{\mathbf{t}(\ve,s)},
\]
respectively.
Hence the curvature of $\vc_{\ve}$ is given by 
\begin{align*}
\add{\frac{\mathbf{n}(\ve, s)\cdot \mathbf{t}'(\ve, s)}{\ell(\ve,s)}} =
\frac{-\pd^2_s\phi (\ve,s) \pd_s\psi  (\ve,s)+ \pd_s\phi  (\ve,s) \pd^2_s\psi (\ve,s)   }{\ell(\ve,s)^3} 
=: \frac{m(\ve,s) }{\ell(\ve,s)^3} ,
\end{align*}
where $'$ denotes the derivative with respect to the parameter $s$.
Therefore we can identify the elastic  energy of $\vc_{\ve}$ with
\[ \W(\vc_{\ve})= \int_0^L m(\ve,s)^2 \ell(\ve,s)^{-5} ds \] 
and it holds that 
\Add{
\begin{align} \label{eq:2.2}
\frac{d}{d\ve} \W(\vc_{\ve}) 
& = \int_0^L \bigg\{ -5\ell^{-7}\Big(\pd_s \phi \cdot  \pd_{\ve} \pd_s \phi +\pd_s \psi \cdot \pd_{\ve} \pd_s \psi  \Big)m^2  \notag  \\
&\quad +2\ell^{-5}m \Big(  \pd_{\ve}\pd_s^2\psi \cdot \pd_s\phi 
 + \pd_s^2\psi\cdot \pd_{\ve}\pd_s\phi  
 -\pd_{\ve}\pd_s^2\phi \cdot \pd_s\psi - \pd_s^2\phi \cdot \pd_{\ve}\pd_s\psi
\Big) \bigg\}\,ds \notag
\\
&= \int_0^L \Add{\pd_s}\bigg\{ 5\ell^{-7}m^2\pd_s\phi -2\ell^{-5}m\,\pd_s^2\psi - 2 \pd_s\Big(
\ell^{-5} m\,\pd_s\psi \Big)\bigg\} \pd_{\ve}\phi \,    ds \notag  \\
&\ \ + \int_0^L \Add{\pd_s} \bigg\{ 5\ell^{-7}m^2\pd_s\psi +2\ell^{-5}m\,\pd_s^2\phi +2 \pd_s\Big(
\ell^{-5} m\,\pd_s\phi \Big)\bigg\} \pd_{\ve}\psi \,    ds  \\
&\ \ + 2\bigg[ \ell^{-5}m\,\pd_s\phi ( \pd_{\ve}\pd_s\psi) \bigg]_0^L
 -2\bigg[ \ell^{-5}m\,\pd_s\psi ( \pd_{\ve}\pd_s\phi) \bigg]_0^L, \notag
\end{align}
}
where we used \Erase{the integral}\Add{integration} by parts and $\pd_{\ve}\phi = \pd_{\ve}\psi =0$ on $\{0,L\}$.
On the other hand, the first variation formula of $\LL$ is given by 
\Add{
\begin{align*}
\frac{d}{d\ve} \LL(\vc_{\ve})  &= \frac{d}{d\ve}\int_0^L \ell(\ve,s) \,ds  \\
&=\int_0^L- \bigg\{  \pd_s\Big(\ell^{-1} \pd_s\phi  \Big) \pd_{\ve}\phi  
+  \pd_s\Big(\ell^{-1} \pd_s\psi  \Big) \pd_{\ve}\psi    \bigg\} ds.
\end{align*}
}
According to the Lagrange multiplier method, $\vc$ is a critical point of $\W$ in $\pinA$ if and only if
there exists some $\lm\in\R$ such that
\begin{align} \label{eq:2.3}
\frac{d}{d\ve} \Big( \W(\vc_{\ve}) + \lm \LL(\vc_{\ve}) \Big) \bigg|_{\ve=0} =0.
\end{align}
By recalling that $\ell(0,s)\equiv1$ and $m(0,s)=\vk(s)$, 
\add{and
restricting $\phi(\ve, \cdot)$ and $\psi(\ve, \cdot)$ to $C^{\infty}_{\rm c}(0,L)$ in \eqref{eq:2.3},
}
we infer from \eqref{eq:2.3} that
\begin{align} \label{eq:2.5}
\begin{split}
&\big( 5\vk^2 X' -4 \vk Y'' -2 \vk'  Y' \big)' -\lm X'' = 0,  \\
&\big( 5\vk^2 Y' +4 \vk X'' +2 \vk'  X' \big)' -\lm Y'' = 0, 
\end{split}
\end{align}
where $'$ denotes the derivative with respect to the parameter $s$.
\add{Then \eqref{eq:2.3} is reduced to
\begin{align}\label{eq:0512-1}
\left(
2\bigg[ \ell^{-5}m\,\pd_s\phi ( \pd_{\ve}\pd_s\psi) \bigg]_0^L
 -2\bigg[ \ell^{-5}m\,\pd_s\psi ( \pd_{\ve}\pd_s\phi) \bigg]_0^L 
 \right)\bigg|_{\ve=0}=0
\end{align}
for all $ \phi, \psi \in C^{\infty}((-\ve_0, \ve_0)\times [0,L])$.
Then choosing $ \phi, \psi $ as functions which satisfy
\[
\pd_s \phi(0,\ve)=\ve, \quad \pd_s \phi(L,\ve)=X'(L), \quad 
\pd_s \psi(0,\ve)=1, \quad \pd_s \psi(L,\ve)=Y'(L), 
\]
we infer from \eqref{eq:0512-1} that $\vk(0)=0$.
Similarly we also obtain $\vk(L)=0$.
Thus we see that any critical point $\vc$ of $\W$ in $\pinA$ satisfies the condition
}
\begin{align} \label{eq:BC}
\vk(0)= \vk(L)=0.
\end{align}
Integrating \eqref{eq:2.5}, we obtain 
\begin{align*}
\big( 5\vk^2-\lm \big) X' -4 \vk Y'' -2 \vk'  Y' = c_1, \quad
\big( 5\vk^2-\lm \big) Y' +4 \vk X'' +2 \vk'  X' = c_2
\end{align*}
for some $c_1, c_2 \in\R$.
This together with $X''=-\vk Y'$ and $Y''=\vk X'$ yields 
\Add{
\[
(\vk^2-\lm) X' -2\vk' Y' =c_1, \quad (\vk^2-\lm) Y' + 2 \vk' X' =c_2.
\]
Solving the above equations for $X'$ and $Y'$ respectively, we obtain 
}
\begin{align} \label{eq:2.7}
\begin{split}
& \Big\{ (\vk^2 - \lm)^2 + (2\vk')^2 \Big\} X' = c_1(\vk^2 - \lm) + 2\vk' c_2, \\
& \Big\{ (\vk^2 - \lm)^2 + (2\vk')^2 \Big\} Y' = -2\vk' c_1 + c_2(\vk^2 - \lm) .
\end{split}
\end{align} 
Since $s$ denotes the arc length parameter of $\vc$, it follows that $X'(s)^2 + Y'(s)^2=1$ and hence we obtain
\begin{align}\label{eq:0308-1} 
\Big\{ (\vk^2 - \lm)^2 + (2\vk')^2 \Big\} ^2 = (c_1^2 + c_2^2) \Big\{ (\vk^2 - \lm)^2 + (2\vk')^2 \Big\} .
\end{align}
\Erase{
Here $(\vk^2 - \lm)^2 + (2\vk')^2=0$ never occurs, since if it holds, $\vc$ is a segment by \eqref{eq:BC}
} 
\Erase{ 
and we obtain the contradiction.
Therefore 
 $(\vk^2 - \lm)^2 + (2\vk')^2= c_1^2 + c_2^2$ follows.
 }
\Add{
Here it follows that $c_1^2 + c_2^2>0$. 
In fact, if $c_1^2 + c_2^2=0$ holds, we infer from \eqref{eq:0308-1} that $(\vk^2 - \lm)^2 + (2\vk')^2\equiv0$.
Then we notice that $\vk^2 \equiv \lm$, which implies $\vc$ is a line segment or a circle.
This contradicts $\vc \in \pinA$ and hence  $c_1^2 + c_2^2 >0$ follows.
Since $\vc\in\pinA$ implies that $(\vk(s)^2-\lm)^2 + 4(\vk'(s))^2$ is continuous on $[0,L]$, we deduce from \eqref{eq:0308-1} that
\[ (\vk^2 - \lm)^2 + (2\vk')^2 \equiv 0 \ \ \text{or}\ \ c_1^2 + c_2^2. \]
Recalling that $(\vk^2 - \lm)^2 + (2\vk')^2 \equiv 0$ does not occur, we conclude that 
\[ (\vk^2 - \lm)^2 + (2\vk')^2 \equiv  c_1^2 + c_2^2. \]
}
Combining this with \eqref{eq:2.7}, we have 
 \begin{align}
 X'(s) &= \frac{c_1}{c_1^2 + c_2^2} \Big( \vk(s)^2 - \lm \Big) + \frac{2c_2}{c_1^2 + c_2^2} \vk'(s),  \label{eq:2.9}\\
 Y'(s) &=- \frac{2\Add{c_1}}{c_1^2 + c_2^2} \vk'(s) +\frac{c_2}{c_1^2 + c_2^2} \Big( \vk(s)^2 - \lm \Big). \label{eq:2.91}
 \end{align}
Since $\vc\in\pinA$ implies that $Y(0)=Y(L)=0$, integrating \eqref{eq:2.91} and using \eqref{eq:BC}, we obtain
 \[ 
 0= \frac{c_2}{c_1^2 + c_2^2} \int_0^L \Big( \vk(s)^2 - \lm \Big)ds, 
 \]
which gives $c_2=0$. 
In fact, if not, then $\int_0^L ( \vk(s)^2 - \lm )ds=0$ holds. 
By \eqref{eq:2.9} it holds that 
\add{$l=0$,}
which contradicts 
\add{$l>0$}.
Therefore for any critical point $\vc$ of $\W$ in $\pinA$, its arc length reparameterization $\tilde{\vc}(s)=(X(s),Y(s))$ can be represented by
\begin{align} \label{eq:2.10}
X(s)= \frac{1}{c_1}\left( \int_0^s \vk(t)^2\,dt -\lm s  \right), \quad
Y(s)=-\frac{2}{c_1} \vk(s).
\end{align}
\Erase{
Moreover, we find that \eqref{eq:EL} follows from the formulas \eqref{eq:2.10}.}
\Add{
This together with $Y''=\vk X'$ implies that
\[ \vk'' = -\frac{c_1}{2}Y'' = -\frac{c_1}{2}\vk X' = -\frac{1}{2}\vk (\vk^2-\lm).   \]
Therefore any critical point $\vc\in \pinA$ satisfies \eqref{eq:EL}.
}

\subsection{Jacobi elliptic functions} \label{subsect:2.2}
Here we collect the notations and facts related to the Jacobi elliptic functions.
Let $\K(p)$ and $\E(p)$ be the complete elliptic integral of the first and second kind, i.e., 
\[ \K(p):=\int_0^1 \frac{1}{\sqrt{1- z^2}\sqrt{1-p^2 z^2}}\,dz, \quad 
\E(p):=\int_0^1\frac{\sqrt{1-p^2  z^2}}{\sqrt{1-z^2}}\,dz,
 \]
for $0\leq p <1$.

\begin{proposition}  \label{prop:A.2}
The function $p\mapsto \K(p)$ is 
\add{monotonically}
increasing and $p\mapsto \E(p)$ is 
\add{monotonically} decreasing. 
Moreover, it holds that
\begin{align*}
\frac{d}{dp} \K(p)= \frac{\E(p)}{(1-p^2)p} - \frac{\K(p)}{p}, \quad
\frac{d}{dp} \E(p)= \frac{\E(p)}{p} - \frac{\K(p)}{p},
\end{align*}
for $p\in(0,1)$.
\end{proposition}
The above formulas are well known so we omit the proof
(\Erase{e.g., see}\Add{see e.g.} \cite[p.282]{Byrd}).

\begin{lemma}\label{lem:2.1}
Let $\vp$ 
\add{$:[0,1) \to \R$ be the function}
defined by 
$
\vp(p):=\E(p) / \K(p) .
$ 
Then $\vp(p)$ is 
\add{monotonically}
decreasing, i.e., it holds that
\begin{align*}
\frac{d \vp}{dp}(p)<0 \quad \text{for} \quad p\in(0,1).
\end{align*}
Moreover, $\vp(0)=1$ and  $\lim_{p\uparrow1}\vp(p)=0$.
\end{lemma}
This clearly holds since it follows from Proposition~\ref{prop:A.2} that
\[ 
\vp'(p) = \frac{\E'(p)\K(p)- \E(p)\K'(p)}{\K(p)^2}<0 \quad \text{for}\quad 0<p<1.
\]
Moreover, $\E(0)=\K(0)=\pi/2$ yields $\vp(0)=1$ and combining $\E(1)=1$ with $\lim_{p\uparrow1}\K(p)=\infty$, we obtain $\lim_{p\uparrow1}\vp(p)=0$.

Next, we mention some basic properties of Jacobi's elliptic functions $\sn$, $\cn$, $\dn$.
\Erase{
We define Jacobi's $\sn$ function by}
\Erase{
for $x\in [-\K(p), \K(p)]$, and $\sn(x,p)=-\sn(x+2\K(p),p)$ for $x\in\R$.
} 
\Erase{
Jacobi's $\cn$, $\dn$ functions are defined by 
}
\Erase{
The modulus $p$ is a number in $(0,1)$.
}
\Add{The elliptic integral of the first kind is defined by
\[
x(\phi):=\int_0^{\phi} \frac{d\theta}{\sqrt{1-p^2 \sin^2\theta}}, \quad 0\leq p \leq 1.
\]
Denoting the inverse of $x(\phi)$ by $\mathrm{am}(x,p)=\phi$, the Jacobi elliptic functions are given by
\[
\sn\,x=\sn(x,p)=\sin \phi, \quad \cn\,x=\cn(x,p)=\cos\phi
\]
and
\[
\dn\,x = \dn(x,p) = \sqrt{1-p^2 \sin^2 \phi}.
\]
}
The function $\cn(x,p)$ is $2\K(p)$-antiperiodic, i.e., \add{$\cn(x+2\K(p),p)= -\cn(x,p)$} for $x\in\R$ and this together with $\cn(\K(p),p)=0$ gives
\begin{align} \label{eq:2.11}
 \cn(x,p)=0  \iff  x=(2n+1)\K(p) \quad  \text{for} \quad n\in\Z. 
\end{align}
Moreover, for $0<p<1$ the following differential formula holds:
\begin{align} \label{eq:A.1}
\frac{d}{dx} \cn(x,p) = -\sn (x,p)\dn(x,p).
\end{align}

\begin{lemma}\label{lem:2.2}
For each $p\in(0,1)$ it holds that
\begin{align} \label{eq:2.12}
\int_0^{\K(p)} \cn(x,p)^2 \,dx = \frac{p^2\K(p)-\K(p) +\E(p)}{p^2}.
\end{align}
\end{lemma}
\begin{proof}
Putting $\cn(x,p)=\xi$ in the left hand side of \eqref{eq:2.12}, we infer from \eqref{eq:A.1} that 
\begin{align*}
\int_0^{\K(p)} \cn(x,p)^2 \,dx = \int_1^0 \xi^2 \frac{(-1)}{\sqrt{1-\xi^2}\sqrt{1-p^2(1-\xi^2)}} \,d\xi ,
\end{align*}
where we used $\sn=\sqrt{1-\cn^2}$, $\dn=\sqrt{1-p^2\sn^2}$ in $(0,\K(p))$.
Then by the change of variable $\sqrt{1-\xi^2}=\zeta$ we have 
\begin{align*}
\int_0^{\K(p)} \cn(x,p)^2 \,dx &= \int_0^1 \frac{\sqrt{1-\zeta^2}}{\sqrt{1-p^2\zeta^2}} \,d\zeta  \\
&= \left(1-\frac{1}{p^2}\right) \int_0^1 \frac{1}{\sqrt{1-\zeta^2}\sqrt{1-p^2\zeta^2}} \,d\zeta 
+\frac{1}{p^2} \int_0^1 \frac{\sqrt{1-p^2\zeta^2}}{\sqrt{1-\zeta^2}} \,d\zeta,
\end{align*}
which with the definitions of $\E(p)$ and $\K(p)$ yields \eqref{eq:2.12}.
\end{proof}

By Proposition~\ref{prop:A.2}, we obtain the \Erase{information}\Add{derivative} of the right hand side of \eqref{eq:2.12} as follows. 
Since the straightforward calculation yields Lemma~\ref{lem:2.3}, we omit the proof.
\begin{lemma}\label{lem:2.3}
\begin{align*}
\frac{d}{dp} \Big( p^2\K(p)-\K(p) +\E(p)  \Big) = \Erase{\frac{\K(p)}{p}} \add{p \K(p)} >0 
\quad \text{for} \quad p\in(0,1).
\end{align*}
\end{lemma}


\section{Representation formulas} \label{section:3}
In this section we shall deduce the representation formulas for critical curves. 
To this end, we consider two cases in subsections~\ref{subsect:3.1} and \ref{subsect:3.2} respectively 
and then Theorem~\ref{thm:1.1} is shown.
Next we discuss about formulas of elastic energy in subsection~\ref{subsect:3.3} and then Theorem~\ref{thm:1.2} is shown.

By the argument in subsection~\ref{subsect:2.1}, critical curves satisfy \eqref{eq:EL} and \eqref{eq:2.10} so we focus on them hereafter.
In order to solve the boundary value problem \eqref{eq:EL}, we employ the shooting method.
\add{Let $\vc \in \pinA$ be a critical point of $\W$. 
It follows from \eqref{eq:EL} that the curvature $\vk$ of $\vc$ satisfies 
\begin{align}\label{eq:ini}
\begin{cases}
\pd_s^2\vk+\dfrac{1}{2}\vk^3-\dfrac{\lm}{2}\kappa = 0, \\
\vk(0)=0, \\
\vk'(0)=b, 
\end{cases}
\end{align}
for some $b, \lm\in \R$.
Let  $\vk_{b, \lm}$ be the unique solution of \eqref{eq:ini} for each $b, \lm\in \R$.
Using $\vk:=\vk_{b, \lm}$ in \eqref{eq:2.10}, we can deduce that $\tilde{\vc}=(X(s), Y(s))$, the arc length parameterization of $\vc$, is written by $(X(s), Y(s))=(X_{b,\lm,c}(s), Y_{b,\lm,c}(s))$, where
\begin{align} \label{eq:0512-2}
\begin{split}
X_{b,\lm,c}(s):= \frac{1}{c}\left( \int_0^s \vk_{b, \lm}(t)^2\,dt -\lm s  \right), \ \ 
Y_{b,\lm,c}(s):=-\frac{2}{c} \vk_{b, \lm}(s),
\end{split}
\end{align}
and $c\in\R\setminus\{0\}$ is a constant satisfying 
\begin{align} \label{eq:0512-3}
\frac{1}{c^2}\lm^2 + \frac{4}{c^2}b^2=1.
\end{align} 
Equation \eqref{eq:0512-3} was deduced from $| \tilde{\vc}'(s)|  \equiv 1$ and \eqref{eq:ini}.
Here we note that \eqref{eq:0512-2} implies $Y_{b,\lm,c}(L)=0$ is equivalent to $\vk_{b,\lm}(L)=0$.
Set
\[
c^+:= \sqrt{\lm^2+4b^2}, \quad c^-:= -\sqrt{\lm^2+4b^2}
\]
and define 
\[
\hat{X}_{b,\lm}:=X_{b,\lm, c^+}, \quad \hat{Y}_{b,\lm}:=Y_{b,\lm, c^+}, \quad
\check{X}_{b,\lm}:=X_{b,\lm, c^-}, \quad \check{Y}_{b,\lm}:=Y_{b,\lm, c^-}.
\]
Thus it suffices to find $b\in\R$ and $\lm \in \R$ satisfying 
\begin{align} \label{eq:0513-1}
 \hat{X}_{b,\lm}(L) = l, \quad \vk_{b, \lm}(L)=0 
\end{align}
or 
\begin{align} \label{eq:0513-2}
 \check{X}_{b,\lm}(L)=l, \quad \vk_{b, \lm}(L)=0. 
\end{align}
}
According to \cite{Lin96}, $\vk_{b, \lm}$ is given by
\begin{align*}
\vk_{b, \lm}(s) = A\,\cn(\va s+\vb, p), 
\end{align*}
where $A\geq0$, $\va\geq0$, $p\in[0,1]$ and $\vb\in[-\K(p),3\K(p))$ are given by
\begin{align}
& A\,\cn(\vb, p)=0,  \label{eq:3.13} \\
-&A\va\,\sn(\vb, p)\dn(\vb,p)=b,  \label{eq:3.14} \\
 &A^2=4\va^2p^2,  \label{eq:3.15} \\
-&\frac{\lm}{2}=\va^2(1-2p^2). \label{eq:3.16}
\end{align}

\add{
We split the proof of Theorem \ref{thm:1.1} into two subsections.
First, we solve \eqref{eq:0513-1} and then obtain
the representation formulas for $\hat{\vc}^{\pm}_n$. 
On the other hand, solving \eqref{eq:0513-2} is equivalent to deriving the representation formulas for $\check{\vc}^{\pm}_n$.
The difference between $c^+$ and $c^-$ drastically changes the equation on $p$ (see \eqref{eq:3.22} and \eqref{eq:3.222}), and is reflected to the feature of shapes of $\hat{\vc}^{\pm}_n$ and $\check{\vc}^{\pm}_n$  (see Figure~\ref{fig:3}).
We mention the parameter $b$. 
By the uniqueness of solutions to \eqref{eq:ini} we see the following:
\begin{align}
& b=0 \ \ \text{implies} \ \ \vk \equiv 0; \notag \\
& \vk_{b, \lm} = -\vk_{-b, \lm} \quad \text{for any} \quad (b,\lm)\in \R^2. \label{eq:0309-1}
\end{align}
If $b=0$, then any critical curve is only the line segment, which does not satisfy $l<L$.
Therefore hereafter we eliminate the case $b=0$.
Combining \eqref{eq:2.10} with \eqref{eq:0309-1}, we have
\begin{align*} 
(\hat{X}_{-b,\lm}, \hat{Y}_{-b,\lm}) = (\hat{X}_{b,\lm}, -\hat{Y}_{b,\lm}), \quad 
(\check{X}_{-b,\lm}, \check{Y}_{-b,\lm}) = (\check{X}_{b,\lm}, -\check{Y}_{b,\lm})
\end{align*}
Therefore it suffices to consider either $b>0$ or $b<0$.
}

\subsection{
\add{Solutions to \eqref{eq:0513-1}}
} \label{subsect:3.1} 
\add{
In this subsection we first find a solution $(b, \lm) \in (-\infty,0)\times \R$ of \eqref{eq:0513-1}.}
Recalling \eqref{eq:2.11},  
we find that $\vb$ in \eqref{eq:3.13} is either 
\[ \vb = -\K(p) \quad \text{or} \quad \vb=\K(p). \] 
Then since $\sn(\pm\K(p), p)=\pm1$ and $\dn(\pm\K(p), p)=\sqrt{1-p^2}$, 
\Add{by \eqref{eq:3.14} we have
\begin{align} \label{eq:3.17}
\vb = \K(p) \quad \text{if} \quad b<0.
\end{align}
}
Moreover, it follows from \eqref{eq:3.14}--\eqref{eq:3.16} and \eqref{eq:3.17} that $A$, $\va$ and $p$ satisfy 
\begin{align} \label{eq:3.19}
\va^2 = \frac{1}{2} \sqrt{4b^2 + \lm^2}, \quad 
p^2 = \frac{1}{2}+\frac{\lm}{4\va^2}, \quad
A^2= \lm + 2\va^2. 
\end{align}

First, we focus on the condition $\vk_{b,\lm}(L)=0$. 
By \eqref{eq:2.11}, $\vk_{b,\lm}(L)= A\cn(\va L+\K(p), p)=0$ holds if and only if
\begin{align} \label{eq:3.18}
\va L +\K(p) = (2n +3)\K(p) \quad \text{for some} \quad n\in \N\cup\{0\}.
\end{align}
\add{Recall that $\hat{X}_{b,\lm}$ is obtained by replacing $c$ with $c^+=\sqrt{\lm^2+4b^2}$ in \eqref{eq:0512-2}, that is, }
\begin{align}
\add{\hat{X}_{b,\lm}}(s)&=\frac{1}{\sqrt{\lm^2+4b^2}}\left( \int_0^s \vk_{b,\lm}(\xi)^2 \,d\xi -\lm s  \right) .
\label{eq:XY} 
\end{align}
\add{Therefore} the remaining condition $\add{\hat{X}_{b,\lm}}(L)=l$ holds if and only if 
\begin{align}\label{eq:3.21}
\frac{1}{2\va^2}\left( \int_0^L \vk_{b,\lm}(s)^2\,ds -\lm L  \right) = l, 
\end{align}
where we used \eqref{eq:XY} and the relation $2\va^2 = \sqrt{\lm^2+4b^2}$ in \eqref{eq:3.19}.
The integral in \eqref{eq:3.21} is reduced to
\begin{align*}
\int_0^L \vk_{b,\lm}(s)^2 \,ds &=  \int_0^L A^2\cn(\va s +\K(p), p)^2\,ds  \\
 &=  \int_{\K(p)}^{\va L+\K(p)} \frac{A^2}{\va}\cn(\zeta , p)^2\,d\zeta \\
\! &\overset{\eqref{eq:3.18}}{=}   \frac{A^2}{\va}\int_{\K(p)}^{(2n+3)\K(p)}\cn(\zeta , p)^2\,d\zeta \\
 &= \frac{ (2n+2) A^2}{\va}  \int_0^{\K(p)} \cn(\zeta , p)^2\,d\zeta,
\end{align*}
where we used the periodicity of $\cn$ in the last equality.
Since it follows from Lemma~\ref{lem:2.2}, \eqref{eq:3.15}, and \eqref{eq:3.18} that 
\begin{align} \label{eq:3.20}
\begin{split}
\int_0^L \vk_{b,\lm}(s)^2 \,ds &= 8(n+1)\va p^2  \cdot \frac{p^2 \K(p)-\K(p) + \E(p)}{p^2} \\
&= \add{ \frac{16(n+1)^2 \K(p)}{L} \Big( p^2 \K(p)-\K(p) + \E(p) \Big) }.
\end{split}
\end{align}
Therefore we can rewrite \eqref{eq:3.21} into
\[ 
\frac{1}{2\va^2}\left( \add{\frac{16(n+1)^2 \K(p)}{L} \Big( p^2 \K(p)-\K(p) + \E(p) \Big)} -\lm L \right) =l, 
\]
which in combination with \eqref{eq:3.19} and  \eqref{eq:3.18} gives
\[
\frac{2L}{\K(p)}\Big(p^2 \K(p)-\K(p) + \E(p) \Big) -(2p^2-1) L  =l. 
\]
Therefore $p\in [0,1]$ must satisfy
\begin{align} \label{eq:3.22}
2\cdot\frac{\E(p)}{\K(p)} -1  = \frac{l}{L}.
\end{align}
Lemma~\ref{lem:2.1} implies that such $p$ is uniquely determined and 
let $\hat{p}$ denote the solution of \eqref{eq:3.22}.
Then plugging $\hat{p}$ into \eqref{eq:3.19} and $\va L=2(n+1)\K(p)$, we notice that $b$ and $\lm$ satisfying \add{\eqref{eq:0513-1}} are given by $(b, \lm)=(\hat{b}_n, \hat{\lm}_n)$, where
\begin{align}\label{eq:1005-1}
\hat{b}_n:=-\frac{\Add{8(n+1)^2}}{L^2}\K(\hat{p})^2\hat{p} \sqrt{1-\hat{p}^2} ,  \quad
\hat{\lm}_n:=-\frac{\Add{8(n+1)^2}}{L^2}\K(\hat{p})^2 \big(1-2\hat{p}^2 \big) , 
\end{align}
for $n\in\N\cup\{0\}$.

\add{
Let us turn to the case $b>0$. 
Recalling \eqref{eq:0309-1}, we have $(b,\lm)=(-\hat{b}_n, \hat{\lm}_n)$ for each $n\in \N$.
Therefore we can summarize the above arguments as follows:
\begin{theorem} \label{thm:3.1}
The pair $(b,\lm)\in\R^2$ solves \eqref{eq:0513-1} if and only if
\begin{align*} 
(b,\lm) = (\hat{b}_n, \hat{\lm}_n) \ \ \text{or} \ \ (-\hat{b}_n, \hat{\lm}_n) \quad \text{for some} \quad n\in \N\cup\{0\}. 
\end{align*}
\end{theorem}
}

\subsection{
\add{Solutions to \eqref{eq:0513-2}}
}  \label{subsect:3.2} 

\add{
In this subsection we first find a solution $(b, \lm) \in (0,\infty)\times \R$ of \eqref{eq:0513-2}.}
\add{Along the same line as in \eqref{eq:3.17},}
\Add{it holds that
\begin{align*} 
\vb = -\K(p) \quad \text{if}\quad b>0, 
\end{align*}
}
and $A$, $\va$ and $p$ in \eqref{eq:3.14}--\eqref{eq:3.16} need to satisfy \eqref{eq:3.19}.
Then \Erase{similarly to}\Add{by the same argument as in} \eqref{eq:3.18}, \add{$\vk_{b,\lm}(L)=0$} holds if and only if  
\begin{align*}
\va L -\K(p) = (2n+1) \K(p) \quad \text{for some} \quad n \in\N\cup\{0\}.
\end{align*}
\Erase{Similarly}\Add{Similar} to \eqref{eq:3.20}, we obtain
\begin{align} \label{eq:3.25}
\begin{split}
\int_0^L \vk_{b,\lm}(s)^2 \,ds 
&= \add{ \frac{16(n+1)^2 \K(p)}{L}  \Big( p^2 \K(p)-\K(p) + \E(p) \Big) }
\end{split}
\end{align}
and hence the necessary and sufficient condition for \add{$\check{X}_{b,\lm}(L)=l$} is that 
\[  -\frac{1}{2\va^2}\left(
\add{ \frac{16(n+1)^2 \K(p)}{L}  \Big( p^2 \K(p)-\K(p) + \E(p) \Big) }
-\lm L \right) =l.
\]
Therefore $p\in[0,1]$ needs to satisfy
\begin{align} \label{eq:3.222}
-2\cdot\frac{\E(p)}{\K(p)} +1  = \frac{l}{L}.
\end{align}
Lemma~\ref{lem:2.1} implies that such a \add{constant} $p$ is uniquely determined and hence we denote \Erase{$\check{p}$ by}\Add{by $\check{p}$} the solution of \eqref{eq:3.222}.
Hence plugging $\check{p}$ into \eqref{eq:3.19} and $\va L=2(n+1)\K(p)$, we notice that $b$ and $\lm$ satisfying \add{\eqref{eq:0513-2}} are given by $(b, \lm)=(\check{b}_n, \check{\lm}_n)$, where
\begin{align}\label{eq:0513-4}
\check{b}_n:=\frac{\Add{8(n+1)^2}}{L^2}\K(\check{p})^2\check{p} \sqrt{1-\check{p}^2} ,  \quad
\check{\lm}_n:=-\frac{\Add{8(n+1)^2}}{L^2}\K(\check{p})^2 \big(1-2\check{p}^2 \big) , 
\end{align}
for $n\in\N\cup\{0\}$.
\add{By considering the case of $b<0$ as well, we obtain the following:
\begin{theorem} \label{thm:3.2}
The pair $(b,\lm)\in\R^2$ solves \eqref{eq:0513-2} if and only if
\begin{align*} 
(b,\lm) = (\check{b}_n, \check{\lm}_n) \ \ \text{or} \ \ (-\check{b}_n, \check{\lm}_n) \quad \text{for some} \quad n\in \N\cup\{0\}. 
\end{align*}
\end{theorem}
\subsection{Characterization of critical points} \label{subsect:3.4} 
To begin with, we shall derive the representation formulas of critical curves.
Set
\begin{align*} 
\hat{\vk}_n&:= \vk_{\hat{b}_n, \hat{\lm}_n}, \quad 
\hat{X}_n:=\hat{X}_{\hat{b}_n, \hat{\lm}_n}, \quad 
\hat{Y}_n:=\hat{Y}_{\hat{b}_n, \hat{\lm}_n},  \\
\check{\vk}_n&:= \vk_{\check{b}_n, \check{\lm}_n}, \quad 
\check{X}_n:=\check{X}_{\check{b}_n, \check{\lm}_n},\quad  
\check{Y}_n:= \check{Y}_{\check{b}_n, \check{\lm}_n}.
\end{align*}
\begin{proof}[Proof of Theorem~\ref{thm:1.1}]
Let $\vc \in \pinA$ and $\tilde{\vc}$ be its arc length parameterization.
By the previous arguments, $\vc \in \pinA$ is a critical curve if and only if 
\[
\hat{\vc}^{+}_n \ \text{or} \ \hat{\vc}^{-}_n \ \text{or} \ \check{\vc}^{+}_n \ \text{or} \ \check{\vc}^{-}_n 
\quad \text{for some} \ \ n\in\N\cup\{0\},
\]
where $\hat{\vc}^{\pm}_n(s):=(\hat{X}_n(s), \pm\hat{Y}_n(s))$ and $\check{\vc}^{\pm}_n(s):=(\check{X}_n(s), \pm\check{Y}_n(s))$.
Moreover, it follows from \eqref{eq:1005-1} and \eqref{eq:0513-4} that 
\begin{align*}
 \hat{\vk}_n(s)&= \frac{4(n+1)}{L}\hat{p}\K(\hat{p})\, \cn\left(\frac{2(n+1)}{L}\K(\hat{p}) s+\K(\hat{p}), \hat{p}\right), \\
 \hat{X}_n(s)&= 
2 \hat{p}^2 \int_0^s \cn\bigg( \frac{2(n+1)\K(\hat{p})}{L} t +\K(\hat{p}), \hat{p} \bigg)^2\,dt + \big(1-2\hat{p}^2 \big) s,   \\
\hat{Y}_n(s)&= -\frac{ \hat{p} L}{(n+1)\K(\hat{p})}
 \cn\bigg( \frac{2(n+1)\K(\hat{p})}{L} s+\K(\hat{p}), \hat{p} \bigg),
\end{align*}
for $s\in[0,L]$ and 
\begin{align*}
\check{\vk}_n(s)&= \frac{4(n+1)}{L}\check{p}\K(\check{p})\, \cn\left(\frac{2(n+1)}{L}\K(\check{p}) s-\K(\check{p}), \check{p}\right), \\
\check{X}_n(s)&= 
-2 \check{p}^2 \int_0^s \cn\bigg( \frac{2(n+1)\K(\check{p})}{L} t -\K(\check{p}), \check{p} \bigg)^2\,dt - \big(1-2\check{p}^2 \big) s,   \\
\check{Y}_n(s)&= \frac{ \check{p} L}{(n+1)\K(\check{p})}
 \cn\Big( \frac{2(n+1)\K(\check{p})}{L} s-\K(\check{p}), \check{p} \Big), 
\end{align*}
for $s\in[0,L]$.
Here, $\hat{p}$ (resp. $\check{p}$) is the unique solution of \eqref{eq:3.22} (resp.  \eqref{eq:3.222}).
The proof is now complete.
\end{proof}
}

Thanks to this representation formula, we can identify what these critical curves are.
\add{To begin with, 
from the periodicity of critical curves $\hat{\vc}_n^{\pm}$ and $\check{\vc}_n^{\pm}$,
one may deduce that $\hat{\vc}_n^{\pm}$ and $\check{\vc}_n^{\pm}$ with $n\in \N \cup\{0\}$ can be constructed from $\hat{\vc}_0^{\pm}$ and $\check{\vc}_0^{\pm}$ respectively.
Indeed we have:
\begin{lemma}\label{lem:5.16}
Let $n\in \N \cup\{0\}$ and $m = 0, 1, \ldots, n$ be arbitrary.
Then 
\begin{align*}
&X_n \big(s+\tfrac{m}{n+1}L\big)= \frac{1}{n+1} X_0\big( (n+1)s \big) + \frac{m}{n+1}l, \\
&Y_n \big(s+\tfrac{m}{n+1}L\big)= (-1)^{m} \frac{1}{n+1} Y_0\big( (n+1)s \big), 
\end{align*}
for $s\in[0, L/(n+1)]$.
Here $(X_n, Y_n)$ is either $(\hat{X}_n, \hat{Y}_n)$ or $(\check{X}_n, \check{Y}_n)$.
\end{lemma}
\begin{proof}
It suffices to prove the case $(X_n, Y_n)=(\hat{X}_n, \hat{Y}_n)$.
Fix $n\in \N \cup\{0\}$, $m = 0, 1, \ldots, n$, and $s\in[0,L/(n+1)]$ arbitrarily.
First, 
\begin{align*}
\hat{Y}_n \big(s+\tfrac{m}{n+1}L\big)
&= -\frac{ \hat{p} L}{(n+1)\K(\hat{p})}
 \cn\bigg( \frac{2(n+1)\K(\hat{p})}{L} s+ (2m+1)\K(\hat{p}), \hat{p} \bigg) \\
&= -\frac{ \hat{p} L}{(n+1)\K(\hat{p})} (-1)^m
 \cn\bigg( \frac{2(n+1)\K(\hat{p})}{L} s+ \K(\hat{p}), \hat{p} \bigg)  \\
&=(-1)^{m} \frac{1}{n+1} \hat{Y}_0\big( (n+1)s \big),
\end{align*}
where we used the fact that $\cn(\cdot, p)$ is $2K(p)$-antiperiodic. 
Next, the periodicity of $\cn$ and the change of variation $u=(n+1)t$ yield 
\begin{align*}
\int_{\frac{m}{n+1}L}^{s+\frac{m}{n+1}L} & \cn\bigg( \frac{2(n+1)\K(\hat{p})}{L} t +\K(\hat{p}), \hat{p} \bigg)^2\,dt \\
&= \frac{1}{n+1}\int_0^{(n+1)s} \cn\bigg( \frac{2\K(\hat{p})}{L} u +\K(\hat{p}), \hat{p} \bigg)^2\,du .
\end{align*}
On the other hand, using the change of variation $u=(n+1)t$ again, we have
\begin{align*}
\int_0^{\frac{m}{n+1}L} & \cn\bigg( \frac{2(n+1)\K(\hat{p})}{L} t +\K(\hat{p}), \hat{p} \bigg)^2\,dt  \\
&=\frac{1}{n+1} \int_0^{mL} \cn\bigg( \frac{2\K(\hat{p})}{L} u +\K(\hat{p}), \hat{p} \bigg)^2\,du \\
&= \frac{m}{n+1} \int_0^{L} \cn\bigg( \frac{2\K(\hat{p})}{L} u +\K(\hat{p}), \hat{p} \bigg)^2\,du,
\end{align*}
where the periodicity of $\cn^2$ is used. 
Therefore we have
\begin{align*}
 \hat{X}_n\big(s+\tfrac{m}{n+1}L\big)&= 
2\hat{p}^2 \bigg[ \frac{m}{n+1} \int_0^{L} \cn\bigg( \frac{2\K(\hat{p})}{L} u +\K(\hat{p}), \hat{p} \bigg)^2\,du \\
&\quad \quad +\frac{1}{n+1}\int_0^{(n+1)s} \cn\bigg( \frac{2\K(\hat{p})}{L} u +\K(\hat{p}), \hat{p} \bigg)^2\,du \bigg] \\
&\quad \quad+ \big(1-2\hat{p}^2 \big) \big(s+\tfrac{m}{n+1}L\big)\\ 
&= \frac{m}{n+1} \hat{X}_0(L) + \frac{1}{n+1} \hat{X}_0\big( (n+1)s \big) .
\end{align*}
Since $\hat{X}_0(L)=l$, the conclusion follows.
\end{proof}
Next, we check the symmetry of $\hat{\vc}_0$ and $\check{\vc}_0$:
\begin{lemma}\label{lem:5.15}
Let $(X_0, Y_0)$ be either $(\hat{X}_0, \hat{Y}_0)$ or $(\check{X}_0, \check{Y}_0)$.
Then  
\begin{align*}
&X_0(s)= l-X_0(-s+L), \quad Y_0(s)= Y_0(-s+L),
\end{align*}
for $s\in[L/2, L]$.
\end{lemma}
\begin{proof}
Fix $s\in[L/2, L]$ arbitrarily.
It suffices to show $\hat{X}_0(s)= -\hat{X}_0(-s+L) + l$, $\hat{Y}_0(s)= \hat{Y}_0(-s+L)$
since the equations for $\check{X}_0$ and $\check{Y}_0$ can be deduced by the same argument.
Since $\cn$ is the even and $2\K(p)$-antiperiodic function, it follows that
\begin{align} \label{eq:0515-1}
\begin{split}
 \cn\Big( \frac{2\K(\hat{p})}{L} (L-s)+\K(\hat{p}), \hat{p} \Big) 
 &=  \cn\Big( -\frac{2\K(\hat{p})}{L} s + 3\K(\hat{p}), \hat{p} \Big) \\
 & =  \cn\Big( \frac{2\K(\hat{p})}{L} s - 3\K(\hat{p}), \hat{p} \Big) \\
 & =  \cn\Big( \frac{2\K(\hat{p})}{L} s + \K(\hat{p}), \hat{p} \Big).
 \end{split}
\end{align}
Therefore $\hat{Y}_0(s)= \hat{Y}_0(-s+L)$ follows from \eqref{eq:0515-1}.
Moreover, \eqref{eq:0515-1} and the change the variable $u=-t+L$ yield
\begin{align*}
\int_L^{-s+L} \cn\bigg( \frac{2\K(\hat{p})}{L} t +\K(\hat{p}), \hat{p} \bigg)^2\,dt 
= -\int_0^{s} \cn\bigg( \frac{2\K(\hat{p})}{L} u +\K(\hat{p}), \hat{p} \bigg)^2\,du,
\end{align*}
from which we obtain 
\begin{align*}
 \hat{X}_0(-s+L)&= 
2 \hat{p}^2 \int_0^{-s+L} \cn\bigg( \frac{2\K(\hat{p})}{L} t +\K(\hat{p}), \hat{p} \bigg)^2\,dt + \big(1-2\hat{p}^2 \big) (-s+L) \\
&=\hat{X}_0(L)- 2 \hat{p}^2 \int_0^{s} \cn\bigg( \frac{2\K(\hat{p})}{L} u +\K(\hat{p}), \hat{p} \bigg)^2\,du -\big(1-2\hat{p}^2 \big) s \\
&= l -  \hat{X}_0(s).
\end{align*}
This completes the proof.
\end{proof}
From now on we focus on the analysis of $\hat{\vc}^+_0$ and $\hat{\vc}^-_0$.}
It follows from \eqref{eq:A.1} that 
\begin{align*}
\frac{d}{ds} \add{\hat{X}_0}(s)&= 
2 \hat{p}^2 \cn\bigg( \frac{2\K(\hat{p})}{L} s +\K(\hat{p}), \hat{p} \bigg)^2 + \big(1-2\hat{p}^2 \big),   \\
\frac{d}{ds} \add{\hat{Y}_0}(s)&= 2\hat{p}\,
 \sn\Big( \frac{2\K(\hat{p})}{L} s+\K(\hat{p}), \hat{p} \Big) \dn\Big( \frac{2\K(\hat{p})}{L} s+\K(\hat{p}), \hat{p} \Big).
\end{align*}
Then we notice that $\add{\hat{X}'_0}(s)$ is \add{monotonically} \Erase{increase}\Add{increasing} in \add{$(0, L/2]$} and satisfies 
\begin{align*}
\add{\hat{X}'_0}(0)\geq0 \quad \text{if} \quad \hat{p} \leq \frac{1}{\sqrt{2}}, \quad
\add{\hat{X}'_0}(0)<0 \quad \text{if} \quad \hat{p}>\frac{1}{\sqrt{2}}.
\end{align*}
Recalling that $\hat{p}$ is defined by \eqref{eq:3.22}, we obtain the following: 
\begin{align*} 
\begin{cases}
\add{\hat{X}'_0}(0)>0 \quad \text{if} \quad R_*<\dfrac{l}{L} < 1, \\
\add{\hat{X}'_0}(0)=0 \quad \text{if} \quad \dfrac{l}{L} = R_*, \\
\add{\hat{X}'_0}(0)<0 \quad \text{if} \quad 0< \dfrac{l}{L} < R_*,
\end{cases}
\end{align*}
where $R_*$ is given by
\begin{align}\label{eq:3.24}
 R_*  = 2\cdot\frac{\E(1/\sqrt{2})}{\K(1/\sqrt{2})} -1 = 0.456946581\ldots.
\end{align}
Moreover, since $\add{\hat{Y}'_0}(s)\geq0$ holds for \add{$s\in[0, L/2]$} (the equality holds if and only if $s=L/2$), 
the curve $\hat{\vc}_n$ can be represented as the graph of a function if \Add{and only if} $l/L > R_*$\Erase{, while $\hat{\vc}_n$  cannot be represented as the graph of any function otherwise} (see Figure~\ref{fig:1}).
\begin{figure}[http]
\centering
\includegraphics[width=12cm]{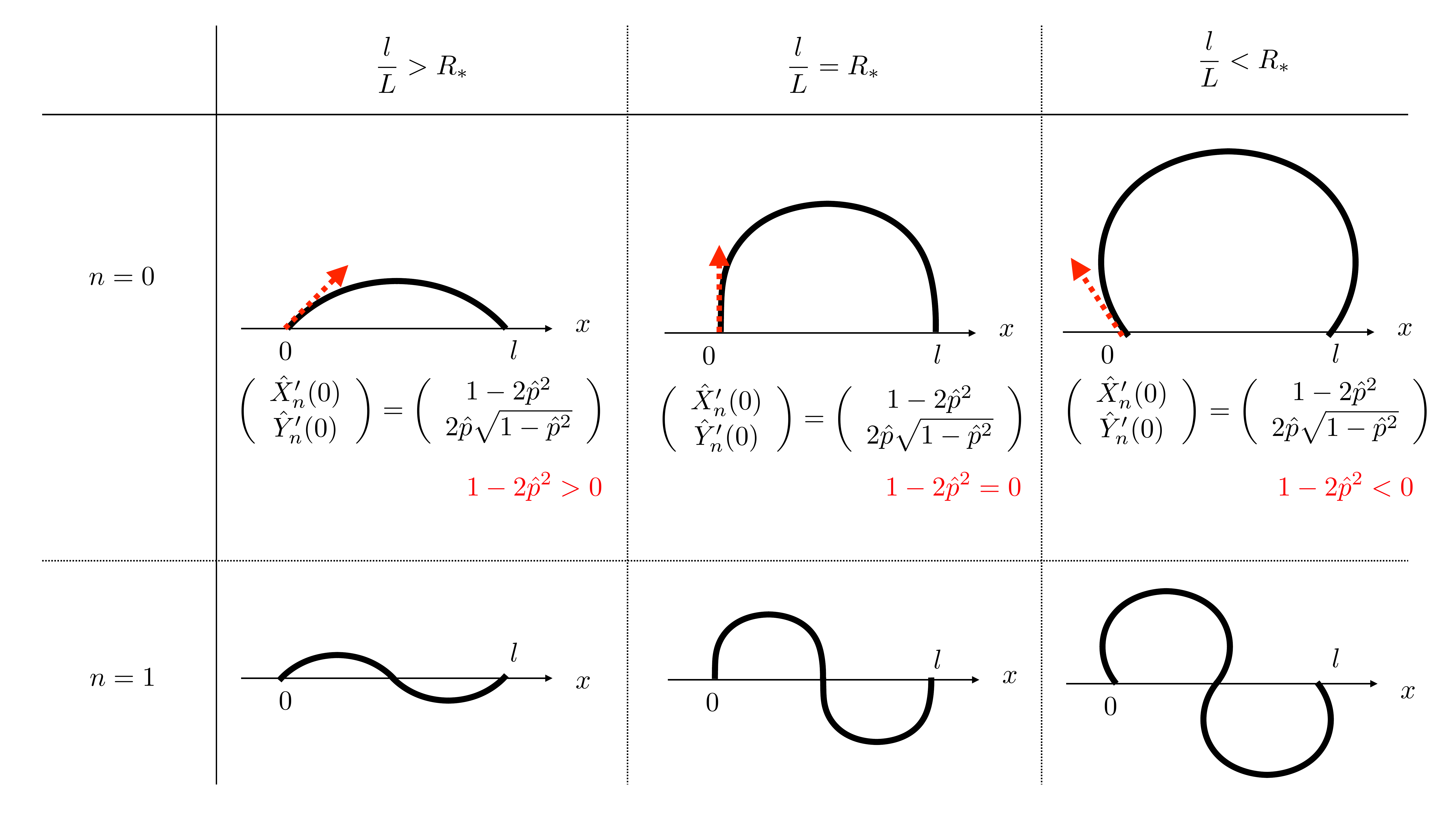} 
\caption{The relation between critical curves and the ratio $l/L$.
The number of inflection points (where the sign of the curvature changes) in $(0,l)$ is given by $n\in\N\cup\{0\}$.
\add{The curve $\hat{\vc}^{\pm}_n$ ($n\in\N$)} \add{can be constructed from $\hat{\vc}^{\pm}_0$.}
}
\label{fig:1}
\end{figure}

\begin{remark}
\add{(1) The curve $\hat{\vc}^+_0=(\hat{X}_0,\hat{Y}_0)$, corresponding to $b<0$, lies in the upper half plane ($\hat{\vc}^-_0=(\hat{X}_0,-\hat{Y}_0)$, corresponding to $b>0$ lies in the lower half plane) so that we choosed $b<0$ in subsection~\ref{subsect:3.1} .
\\
(2)} As explained in \cite[Fig.\,14]{Miura20}, $\hat{\lm}_n$ plays the role of ``tension", that is, critical curves cannot be represented as the graph of any function when $\hat{\lm}_n>0$ holds. 
The values of $\hat{\lm}_n$ are given by \eqref{eq:1005-1}, which implies that the tension value is determined only by the ratio $l/L$.
\end{remark}

\add{
We turn to the analysis of $\check{\vc}^+_0$ and $\check{\vc}^-_0$.}
Define $p_0$ by the constant given by
\[
2\cdot\frac{\E(p_0)}{\K(p_0)} -1 =0.
\]
By \eqref{eq:3.22} and \eqref{eq:3.222}, it holds that
\begin{align} \label{eq:3.30}
 0<\hat{p}<p_0<\check{p}<1 \quad \text{for any ratio}\ \  \frac{l}{L}.
\end{align}
Combining \eqref{eq:3.24} with the monotonicity of $\E(p)/\K(p)$, we notice that $p_0>1/\sqrt{2}$
(actually, numerical simulations show that $p_0=0.90890\ldots$. 
See Figure~\ref{fig:3} \Erase{to know}\Add{for} the relation between $p$ and $l/L$).
Hence 
\begin{align} \label{eq:0516-1}
\check{p}>\frac{1}{\sqrt{2}}
\end{align} 
holds.
\add{
\begin{lemma}\label{lem:5.17}
For each $n\in\N\cup\{0\}$ and $m=0, \ldots, n$, $\check{\vc}^+_n(s)$ and $\check{\vc}^-_n(s)$ have a loop in $ [\frac{m}{n+1}L, \frac{m+1}{n+1}L]$.
\end{lemma}
\begin{proof}
By Lemma~\ref{lem:5.16}, it suffices to show that $\check{\vc}^+_0(s)$ has a loop in $[0,L]$.
Since 
\begin{align*}
\frac{d}{ds} \check{X}_0(s)&= 
-2 \check{p}^2 \cn\bigg( \frac{2\K(\check{p})}{L} s -\K(\check{p}), \check{p} \bigg)^2 - \big(1-2\check{p}^2 \big),
\end{align*}
it follows from \eqref{eq:0516-1} that $\check{X}_0'(0) = 2\check{p}^2-1>0$ and $\check{X}_0' ({L}/{2})=-1<0$.
Moreover, $\cn \big( \frac{2\K(\check{p})}{L} s -\K(\check{p}), \check{p} \big)$ is monotonically increasing 
in $[0,L/2]$ and hence
$ \check{X}_0'(s)$ is monotonically decreasing in $[0,L/2]$.
Together this with $\check{X}_0(0)=0$ and $ \check{X}_0(L/2)=l/2$ implies that
there is $s_* \in (0,L/2)$ satisfying 
\[ \check{X}_0(s_*) = \frac{l}{2}. \]
By Lemma~\ref{lem:5.15} we have 
\[
\check{\vc}^+_0(s_*) = \check{\vc}^+_0(L-s_*).\]
Moreover, it follows that $\check{Y}_0(s)$ is monotonically increasing in $[0,L/2)$ and hence 
we find that $\check{\vc}^+_0(s)$ has a loop in $[0,L]$.
\end{proof}
}

\begin{remark}
\add{By Lemmas~\ref{lem:5.16} and \ref{lem:5.17}, we notice that the curve $\check{\vc}^+_n$ can be represented as in Figure~\ref{fig:2}.} 
\end{remark}

\begin{figure}[http]
\centering
\includegraphics[width=10.5cm]{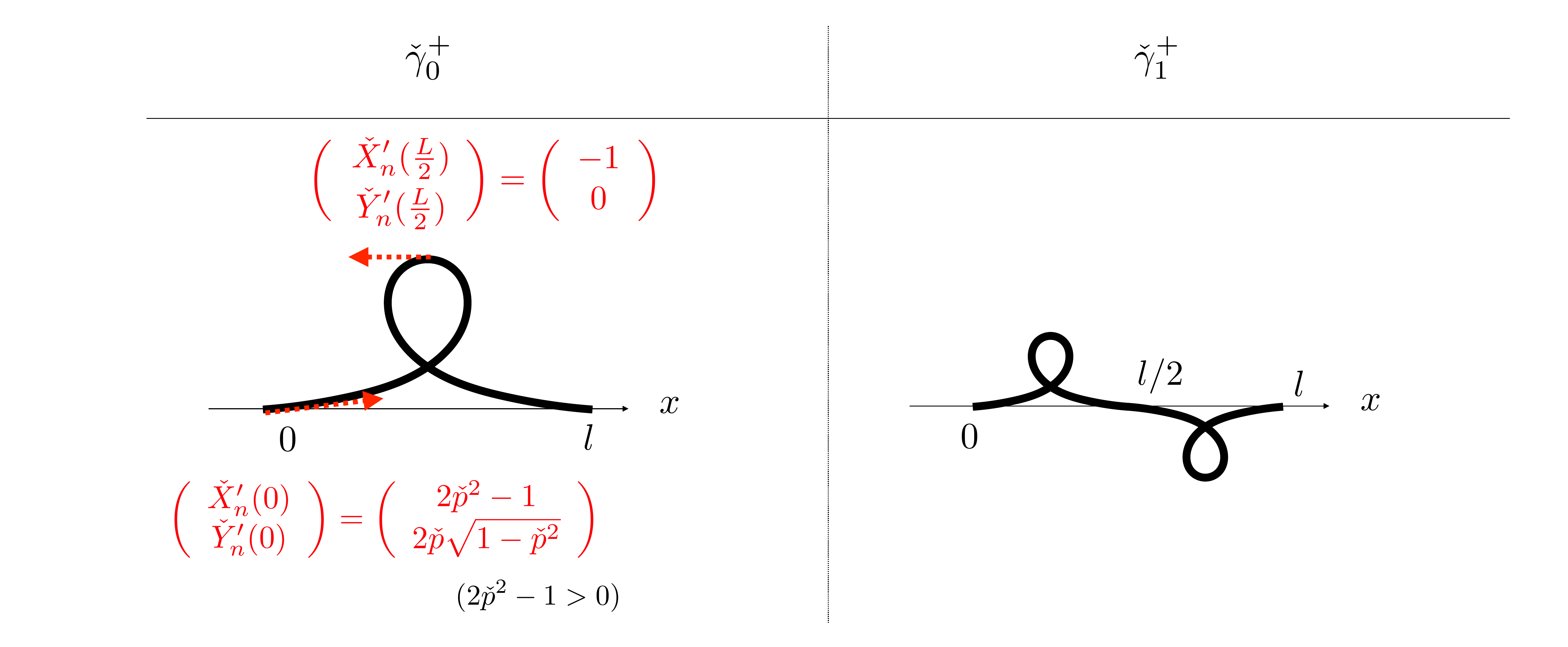} 
\caption{\add{For any $n\in\N\cup\{0\}$ the curves $\check{\vc}^+_n$ and $\check{\vc}^-_n$ has a loop.}}
\label{fig:2}
\end{figure}

\subsection{Comparison among critical curves} \label{subsect:3.3} 
In this subsection we turn to the proof of Theorem~\ref{thm:1.2}
and consider the case that $l/L$ is sufficiently close to $1$.
\add{Since $\W(\hat{\vc}^{+}_n)=\W(\hat{\vc}^-_n)$ and $\W(\check{\vc}^{+}_n)=\W(\check{\vc}^-_n)$ for $n\in\N\cup\{0\}$, hereafter we respectively use $\W(\hat{\vc}_n)$ and $\W(\check{\vc}_n)$ instead of $\W(\hat{\vc}^+_n)$ and $\W(\check{\vc}^+_n)$.}

\begin{proof}[Proof of Theorem~\ref{thm:1.2}]
Fix $0<l<L$ arbitrarily.
By \eqref{eq:3.20} each elastic  energy for $\hat{\vc}_n$ is 
\begin{align} \label{eq:1005-2}
\W(\hat{\vc}_n) =  \frac{16(n+1)^2  \K(\hat{p})}{L}  \Big( \hat{p}^2 \K(\hat{p})-\K(\hat{p}) + \E(\hat{p})  \Big)
\end{align}
and the elastic  energy for $\check{\vc}_n$ is
\begin{align}\label{c-energy}
\W(\check{\vc}_n) =  \frac{16(n+1)^2  \K(\check{p})}{L}  \Big( \check{p}^2 \K(\check{p})-\K(\check{p}) + \E(\check{p})  \Big), 
\end{align}
according to \eqref{eq:3.25}. 
Then we \Erase{soon}\Add{immediately} obtain 
\[
\W(\hat{\vc}_n) = (n+1)^2\W(\hat{\vc}_0), \quad \W(\check{\vc}_n) = (n+1)^2 \W(\check{\vc}_0).
\]
Moreover, \add{since $(p^2\K(p)-\K(p)+\E(p))|_{p=0}$=0,} \Erase{since}it follows from Proposition~\ref{prop:A.2} and Lemma~\ref{lem:2.3} that 
\[ 
p \mapsto \K(p) 
\Big( p^2 \K(p)-\K(p) + \E(p)  \Big) 
\text{ is a \add{monotonically}\Erase{strictly} increasing function on } (0,1).
\]
Recalling \eqref{eq:3.30}, the relation between $\hat{p}$ and $\check{p}$, we \Erase{soon}\Add{then} obtain 
\[ \W(\hat{\vc}_n) < \W(\check{\vc}_n) \quad \text{for each}\quad n\in\N. \]
\Erase{We complete}\Add{This completes} the proof.
\end{proof}

\begin{figure}[http]
\centering
\includegraphics[width=10cm]{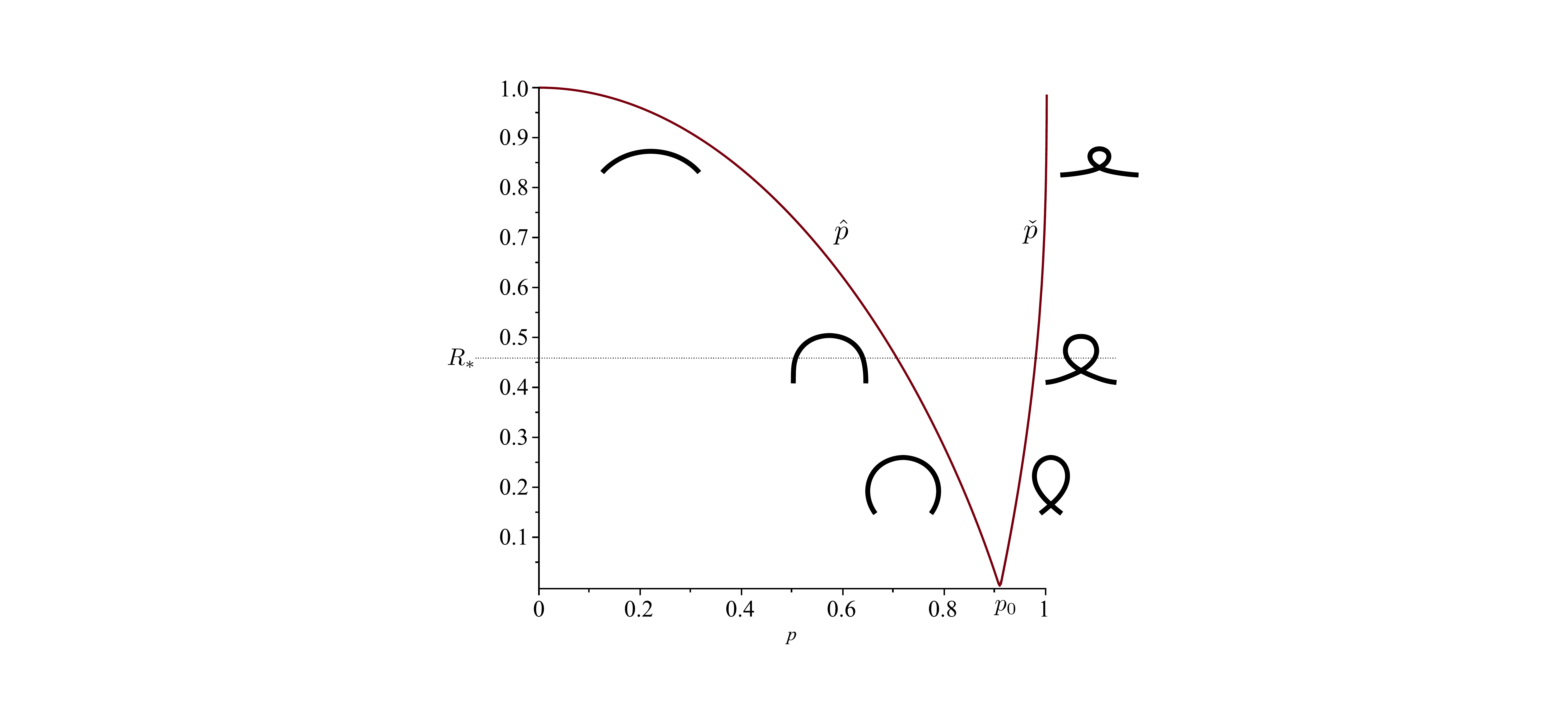} 
\caption{The red curve is the graph of 
$2\frac{\E(p)}{\K(p)}-1$;  $0\leq p \leq p_0$ \ and \ $-2\frac{\E(p)}{\K(p)}+1$;  $p_0< p <1$.
}
\label{fig:3}
\end{figure}

\Add{We are ready to determine which curve is the global minimizer:
\begin{proof}[Proof of Theorem~\ref{thm:1.3}]
The existence of global minimizers of $\W$ in $\pinA$ follows from the direct method in the calculus of variations (see Appendix A).
Since minimizers are also critical points of $\W$ in $\pinA$, in order to identify minimizers it suffices to find critical points whose energy is least.
By Theorem~\ref{thm:1.2}, $\hat{\vc}^{+}_0=(\hat{X}_0, \hat{Y}_0)$ attains the minimum of $\W$ among critical points and so does $\hat{\vc}^{-}_0=(\hat{X}_0, -\hat{Y}_0)$.
Moreover, it has been shown that the other critical points do not attain the minimum. 
This implies that minimizers are nothing but $\hat{\vc}^{\pm}_0$.
\end{proof}
}

By \add{Lemma~\ref{lem:2.1}}, \eqref{eq:3.22}, and \eqref{eq:1005-2}, the \add{larger}\Erase{smaller} the ratio $l/L$ is, the less the elastic energy of $\hat{\vc}^{\pm}_n$ is. 
Conversely, it follows from \add{Lemma~\ref{lem:2.1}}, \eqref{eq:3.222}, \eqref{c-energy} that the \add{smaller}\Erase{larger} the ratio $l/L$ is, the less the elastic  energy of $\hat{\vc}^{\pm}_n$ is.
\Add{In order to investigate $\W(\check{\vc}_0)$ when $l/L$ is close to $1$, we will require the following lemma}:
\begin{lemma}\label{prop:3.5}
Let $\K(p)$ and $\E(p)$ be the complete elliptic integral. 
Then it holds that
\begin{align*}
 \K(p)  \Big( p^2 \K(p)-\K(p) + \E(p)  \Big) \to \infty \quad\text{as}\quad p\uparrow1.
\end{align*}
\end{lemma}
\begin{proof}
To begin with, we notice that for any $p\in(0,1)$ 
\[ 
\sqrt{1-p^2}\, \K(p) = \int_0^1 \frac{1}{\sqrt{1-z^2}} \frac{\sqrt{1-p^2}}{\sqrt{1-p^2z^2}}dz
\leq  \int_0^1 \frac{1}{\sqrt{1-z^2}} dz = \frac{\pi}{2}.
\]
Hence it follows that
\[ \liminf_{p \uparrow 1} (1-p^2) \K(p)^2 \leq \frac{\pi^2}{4}.
\]
This together with $\E(1)=1$ and $\lim_{p\uparrow1}\K(p)=\infty$ yields
\begin{align*}
 \liminf_{p \uparrow 1} \Big( \K(p)  \big( p^2 \K(p)-\K(p) + \E(p) \big)  \Big) 
 \geq - \frac{\pi^2}{4} + \liminf_{p \uparrow 1}\Big( \E(p) \K(p) \Big) = \infty,
\end{align*}
\Erase{We complete}\Add{which completes} the proof.
\end{proof}

Since $l/L\to1$ implies that $\check{p}\to1$, combining Lemma~\ref{prop:3.5} with \eqref{c-energy}, we observe that
\[ \W(\check{\vc}_0) \to \infty \quad\text{as}\quad l/L\to 1. \]

\begin{remark}\label{rem:3.5}
\add{(1)} 
Fix $L>0$ arbitrarily.
It follows from the previous argument that
\[ \bullet \ \  \text{when $l$ is sufficiently close to $L$, $\W(\Add{\check{\vc}^+_0} )>\W(\Add{\hat{\vc}_1})$ holds}. \]
On the other hand,
\[ \bullet \ \ \text{when $l$ is sufficiently small, $\W(\Add{\check{\vc}_0} )<\W(\Add{\hat{\vc}_1})$ holds}.\]
In fact, since both $\hat{p}$ and $\check{p}$ tend to $p_0$ as $l\downarrow 0$, 
 $\W(\hat{\vc}_n)$ is almost equal to $\W(\check{\vc}_n)$ for each $n\in\N$.
This observation implies that
\Add{\textit{$l/L$ determines which of the curves $\hat{\vc}_1$ and $\check{\vc}_0$ has the second smallest elastic energy.}
}\\
(2) \add{In the case $l=0$, the minimization problem is considered by \cite{ANP_X} and the minimizes are uniquely determined as the half-fold figure-eight up to the reflection (see \cite{MiuraX21, MR_X}).
We expect that $\hat{\vc}_0$ and $\check{\vc}_0$ tend to the half-fold figure-eight letting $l\to0$ formally.
}
\end{remark}





\appendix

\section{Existence of minimizers}

Here we shall prove the existence and the smoothness of minimizers of $\W$ in $\pinA$ via a direct method in the calculus of variations.
Hereafter $\W'(\vc)(\vp)$ and $\LL'(\vc)(\vp)$ denote the first variation of $\W$ and $\LL$ at $\vc$ in direction $\vp$, respectively.
That is, 
\[ 
\W'(\vc)(\vp)= \frac{d}{d\ve}\W(\vc+\ve\vp)\Big|_{\ve=0} , \quad \LL'(\vc)(\vp)= \frac{d}{d\ve}\LL(\vc+\ve\vp)\Big|_{\ve=0}.
\]
For $\vc$ parameterized by the arc length, it holds that 
\begin{align}
\W'(\vc)(\vp)&= \int_0^L\bigg( 2\pd_s^2\vc\cdot \pd_s^2\vp -3(\pd_s\vc\cdot \pd_s \vp)|\pd_s^2\vc|^2 \bigg)ds, \label{eq:a.11} \\
\LL'(\vc)(\vp)&=  \int_0^L \pd_s \vc \cdot \pd_s \vp \,ds, \notag
\end{align} 
for $ \vp:[0,L]\to\R^2$,
where $s$ is the arc length parameter of $\vc$ (for \eqref{eq:a.11}, see e.g. \cite{DMN}). 

To begin with, we shall show that the regularity of critical points is improved:
\begin{theorem} \label{thm:a.2} 
All critical points of $\W$ in 
\begin{align*}
H_{l,L} := \Set{ \vc \in H^2\big(0,1;\R^2 \big) | 
\begin{array}{l}
\vc(0)=(0,0),  \ 
\vc(1)= (0,l), \\
\LL(\vc)=L, \ \ \ \ \ 
\vc'(t)\ne0 
\end{array}
}
\end{align*}
belong to $C^{\infty}(0,1)$, up to a reparameterization.
\end{theorem}
\begin{proof}
Let $\vc \in H_{l,L}$ be an arbitrary critical point of $\W$
and $\tilde{\vc}$ be the arc length parameterization of $\vc$. 
Then it clearly holds that $\tilde{\vc}$ is a minimizer of $\W$ in 
\begin{align*}
\tilde{H}_{l,L} := \Set{ \vc \in H^2\big(0,L;\R^2 \big) | 
\begin{array}{l}
\vc(0)=(0,0),  \ 
\vc(L)= (0,l), \\
\LL(\vc)=L, \ \ \ \ \ 
\vc'(t)\ne0 
\end{array}
}.
\end{align*}
Thanks to the Lagrange multiplier method, there exists $\lm\in\R$ such that 
\begin{align*} 
0&=\W'(\tilde{\vc})(\vp)-\lm \LL'(\tilde{\vc})(\vp) \\
&=\int_0^L\bigg( 2\pd_s^2\tilde{\vc}\cdot\vp'' -3(\pd_s\tilde{\vc}\cdot\vp')|\pd_s^2\tilde{\vc}|^2 \bigg)ds -\lm \int_0^L \pd_s\tilde{\vc}\cdot\vp' \,ds
\end{align*}
for any $\vp\in C^{\infty}_{\mathrm{c}}(0,L;\R^2)$.
With the help of a bootstrap argument the smoothness of $\tilde{\vc}$ follows.
Moreover, $\bar{\vc}(t):=\tilde{\vc}(Lt)$ ($t\in[0,1]$), a reparameterization of $\vc$, is also smooth.
\end{proof}


\begin{theorem} \label{thm:a.1} 
The minimization problem
\[
\min_{\vc\in \pinA} \W(\vc)
\]
admits a solution.
\end{theorem}
\begin{proof}
First, we show that $\W$ has a minimizer in $H_{l,L}$. 
Let $\{\vc_k\}_{k\in\N}\subset H_{l,L}$ be a minimizing sequence for $\W$, that is,  
\begin{equation}\label{eq:a.16}
\lim_{k\to\infty}  \W(\vc_k) = \inf_{\vc\in H_{l,L}}  \W(\vc). 
\end{equation}
Then we can find $C>0$ such that $\W(\vc_k) \leq C$ for $k\in\N$.

\smallskip

\noindent
\textbf{Step 1}. {\sl The constant parameterization.\ }
Let $\tilde{\vc}_k(s)$ be the arc length parameterization of $\vc_k$ ($s=s_k$: the arc length of $\vc_k$) and set 
\[
\bar{\vc}_k(x) :=\tilde{\vc}_k(Lx) \quad\text{for}\quad x\in[0,1]
\]
(The arc length parameterization of $\bar{\vc}_k$ is clearly $\tilde{\vc}_k$).
We notice that $\bar{\vc}_k$ belongs to $H_{l,L}$ for $k\in\N$. 
Recalling that $|\pd_s \tilde{\vc}_k(s)|\equiv1$ for $s\in[0,L]$, we obtain
\[
| \bar{\vc}'_k(x) |\equiv L \quad \text{for}\quad x\in[0,1], 
\]
which yields
\begin{align*}
\W({\vc}_k)&=\int_0^L \vk(s)^2 \,ds 
=  \frac{1}{L^3} \int_0^1 \left| \pd_x^2\bar{\vc}_k \right|^2 \,dx
= \W(\bar{\vc}_k).
\end{align*}

\smallskip

\noindent
\textbf{Step 2}. {\sl Show that there is a minimizer in $H_{l,L}$.}
Since $\W(\vc_k) \leq C$ for $k\in\N$, we infer from the above equation that 
\begin{align}\label{eq:a.07}
\int_0^1 \left| \pd_x^2\bar{\vc}_k \right|^2 \,dx \leq C L^3.
\end{align}
On the other hand, by $\mathcal{L}(\bar{\vc}_k)=L$ and $\bar{\vc}_k=(0,0)$, we have $\| \bar{\vc}_k \|_{L^{\infty}(0,1)} \leq L$ and hence
\begin{align}\label{eq:a.008}
\| \bar{\vc}_k \|_{L^{2}(0,1)} \leq L
\end{align}
follows.
Therefore combining $|\bar{\vc}'_{k}| \equiv L$ with \eqref{eq:a.07} and \eqref{eq:a.008}, we obtain
\[
\| \bar{\vc}_k \|_{H^2} = \sqrt{\|\bar{\vc}_k \|^2_{L^2} + \|\bar{\vc}_k' \|^2_{L^2} +\|\bar{\vc}_k'' \|^2_{L^2}  }
\leq \sqrt{\tfrac{2}{3}L^2+ L^2 + L^3 C }.
\]
Then, by the Sobolev compact embedding, there exists a curve $\vc\in H^2(0,1;\R^2)$ and a subsequence $\{ \bar{\vc}_{k_j}\}_j$ such that
\[ 
\bar{\vc}_{k_j} \to \vc
\]
weakly in $H^2(0,1;\R^2)$ and $C^1([0,1];\R^2)$.
Combining the convergence in $C^1$ with $\bar{\vc}_k\in H_{l,L}$, we obtain 
\[
\vc(0) = (0,0), \quad \vc(1)=(l,0), \quad  \LL(\vc)=L,\quad  |\vc'|\equiv L>0,
\]
which implies $\vc \in H_{l,L}$.
Furthermore, thanks to the weak convergence in $H^2$, we infer from $|\vc'|\equiv L$ that
\[
\W(\vc)=\frac{1}{L^3}\int_0^1 |\vc''|^2\,dx \leq  \liminf_{j\to\infty}\frac{1}{L^3}\int_0^1 |\bar{\vc}_{k_j}''|^2\,dx = \inf_{\vc\in H_{l,L}} \W(\vc),
\]
where in the last equality we used \eqref{eq:a.16}.
Therefore $\vc$ is a minimizer of $\W$ in $H_{l,L}$.

\smallskip

\noindent
\textbf{Step 3}. {\sl Conclusion.}
Set
\[
\bar{\vc}(x):=\tilde{\vc}(Lx) \quad \text{for}\quad x\in[0,1],
\]
where $\tilde{\vc}$ is the arc length parameterization of $\vc$.
By Step 2 and Theorem~\ref{thm:a.2}, we find that $\bar{\vc}\in H_{l,L}$ and $\bar{\vc}\in C^{\infty}([0,1];\R^2)$
and hence it holds that $\bar{\vc}\in \pinA$.
Furthermore, since $\vc$ is a minimizer in $H_{l,L}$, we can conclude that $\bar{\vc}$ is also a minimizer in $\pinA$. 
\end{proof}


\subsection*{Acknowledgements}
The author was supported by \Erase{JSPS KAKENHI Grant Number 19J20749}\Add{Grant-in-Aid for JSPS Fellows 19J20749}. 
The author would like to thank Professor Shinya Okabe for fruitful discussions 
\Add{and Philip Schrader at Tohoku University for several helpful comments.
Moreover, the author expresses his gratitude to the reviewer for his great effort and numerous helpful suggestions leading to an improved version of the manuscript.
}


\end{document}